\documentclass[a4paper,leqno]{amsart}

\usepackage{latexsym}
\usepackage[english]{babel}
\usepackage{fancyhdr}
\usepackage[mathscr]{eucal}
\usepackage{amsmath}
\usepackage{mathrsfs}
\usepackage{amsthm}
\usepackage{amsfonts}
\usepackage{amssymb}
\usepackage{amscd}
\usepackage{bbm}
\usepackage{graphicx}
\usepackage{subcaption}
\usepackage{graphics}
\usepackage{latexsym}
\usepackage{color}

\usepackage{pifont}
\usepackage{booktabs} 

\newcommand{\ud}{\mathrm{d}}

\newcommand{\ii}{\mathrm{i}}
\newcommand{\cH}{\mathcal{H}}
\newcommand{\ran}{\mathrm{ran}}

\theoremstyle{plain}
\newtheorem{theorem}{Theorem}[section]
\newtheorem{lemma}[theorem]{Lemma}
\newtheorem{corollary}[theorem]{Corollary}
\newtheorem{proposition}[theorem]{Proposition}

\theoremstyle{definition}

\newtheorem{remark}[theorem]{Remark}
\newtheorem{example}[theorem]{Example}

\numberwithin{equation}{section}

\begin{document}

\title[New Ess-Spec by self-adjoint extension of gapped operators]
{On creating new essential spectrum by self-adjoint extension of gapped operators}
\author[A.~Michelangeli]{Alessandro Michelangeli}
\address[A.~Michelangeli]{Department of Mathematics and Natural Sciences, Prince Mohammad Bin Fahd University \\ Al Khobar 31952 (Saudi Arabia) \\
and Hausdorff Center for Mathematics, University of Bonn \\ Endenicher Allee 60 \\ D-53115 Bonn (Germany)\\ and TQT Trieste Institute for Theoretical Quantum Technologies, Trieste (Italy)}
\email{amichelangeli@pmu.edu.sa}


\begin{abstract}
 Given a densely defined and gapped symmetric operator with infinite deficiency index, it is shown how self-adjoint extensions admitting arbitrarily prescribed portions of the gap as essential spectrum are identified and constructed within a general extension scheme. The emergence of new spectrum in the gap by self-adjoint extension is a problem with a long history and recent deep understanding, and yet it remains topical in several recent applications. Whereas it is already an established fact that, in case of infinite deficiency index, any kind of spectrum inside the gap can be generated by a suitable self-adjoint extension, the present discussion has the virtue of showing the clean and simple operator-theoretic mechanism of emergence of such extensions. 
 \end{abstract}

\date{\today}




\thanks{Partially supported by the Italian National Institute for Higher Mathematics INdAM, the Alexander von Humboldt Foundation, Bonn, and Prince Mohammed Bin Fahd University, Al Khobar. This note is strongly inspired by the very clean and instructive presentations on closely related subjects delivered by M.~Holzmann and by K.~Pankrashkin at the workshop \emph{Spectral Theory of Differential Operators in Quantum Theory} (J.~Behrndt, F.~Gesztesy, A.~Laptev, and C.~Tretter organisers) held at the ESI Vienna on 7-11 November 2022, and by B.~Benhellal at the workshop \emph{Singular Perturbations and Geometric Structures} (M.~Gallone, A.~Michelangeli, and L.~Rizzi organisers) held at SISSA Trieste on 21-23 November 2022. The author is also most grateful to the Erwin Schr\"{o}dinger International Institute for Mathematics and Physics (ESI), Vienna, and to the Mathematical Institute at the Silesian University in Opava (and in particular to N.~A.~Caruso and K.~Has\'ik), for the kind hospitality during the period in which this project was set up.}

\maketitle


\section{Introduction and background}\label{sec:intro}

The general problem of producing prescribed portions and types of additional spectrum by constructing suitable self-adjoint extensions of certain symmetric operators, or also, on a closely related footing, by taking singular perturbations of self-adjoint operators, is a long lasting one in operator theory, both in abstract terms and in application, typically, to differential operators. This is witnessed, beside many fundamental precursors that are going to be briefly surveyed in a moment, by the recent advances made by the group of Prof J.~Behrndt and the group of Prof K.~Pankrashkin, which provided the main inspiration and motivation for the present note and will be illustrated below.

The playground for such a problem are \emph{gapped} operators, namely densely defined symmetric operators $S$ on a complex Hilbert space $\cH$ admitting an open (finite or semi-infinite) interval $J=(a,b)$ as a `gap', in the precise sense that
\begin{equation}\label{eq:gapcondition}
\begin{split}
 \langle f,Sf\rangle\;&\geqslant\;b\,\|f\|^2\qquad\qquad \forall f\in\mathcal{D}(S) \qquad\textrm{if } a\,=\,-\infty\,, \\
 \Big\|\Big(S-\frac{a+b}{2}\mathbbm{1}\Big)f\Big\|\;&\geqslant\;\frac{b-a}{2}\,\|f\| \qquad \;\,\forall f\in\mathcal{D}(S) \qquad \textrm{if } a\,>\,-\infty\,. 
\end{split}
\end{equation}
%
Here and in the following, as customary, the notation $\langle\cdot,\cdot\rangle$ indicates the scalar product of $\cH$ (and $\|\cdot\|$ is the associated norm), with the convention that the anti-linear entry is the left one, and $\mathcal{D}(A)$ denotes the domain of an operator $A$.

As a matter of fact, owing to classical results by Kre{\u\i}n \cite[Theorem 21]{Krein-1947}, Stone \cite[Theorem 9.21]{Stone1932}, Friedrichs \cite{Friedrichs1934}, and Freudenthal \cite{Freudenthal-1936}, the gap condition \eqref{eq:gapcondition} is equivalent to the existence of a self-adjoint extension $\widetilde{S}$ of $S$ having $J$ as a gap in its spectrum, i.e., $\sigma(\widetilde{S})\cap J=\emptyset$. Also, one sees that in the limit of infinite gap the second condition of \eqref{eq:gapcondition}, equivalently re-written as
\[
 \|Sf\|^2-(a+b)\langle f,Sf\rangle + ab\|f\|^2\;\geqslant\;0\,,
\]
takes the form of the first of \eqref{eq:gapcondition} upon dividing by $a<0$ and letting $a\to-\infty$.

 Now, a gapped symmetric operator $S$ whose deficiency indices
 \begin{equation*}
  \dim(\ker(S^*-z\mathbbm{1}))\qquad\textrm{and}\qquad \dim(\ker(S^*-\overline{z}\mathbbm{1}))
 \end{equation*}
 (that, owing to the existence of a gap, are necessarily equal for any $z\in\mathbb{C}^+\cup J$) are \emph{infinite} allows for such a vast variety of inequivalent self-adjoint extensions to make it rather plausible to expect that there \emph{are} extensions of $S$ having any sort of prescribed spectrum within the gap $J$.

 This picture has been actually confirmed, with many related aspects still under investigation, from multiple perspectives:
 \begin{itemize}
  \item by considering the abstract existence or also abstract construction of self-adjoint extensions with prescribed spectrum in the gap \cite{Brasche-Neidhardt-Weidmann-1992,Brasche-Neidhardt-Weidmann-1992-1993,Brasche-Neidhardt-1995-absspectr,Brasche-Neidhardt-1996,Albeverio-Brasche-Neidhardt-JFA1998,Brasche-JOT2000,Brasche2001_ContMath2004,Albeverio-Brasche-Malamud-Neidhardt-JFA2005,Behrndt-Khrabustovskyi-2021};
    \item by considering abstract \emph{singular perturbations} of a self-adjoint operator (in the standard sense, e.g., of \cite{albverio-kurasov-2000-sing_pert_diff_ops}), which display prescribed eigenvalues and eigenvectors \cite{Kondej-2001,Albe-Dudkin-Konst-Koshm-2005,Konstantinov-2005,Albe-Dudkin-Konst-Koshm-2007};
  \item by focussing on explicit gapped differential operators on variable bounded domains of $\mathbb{R}^d$ with given boundary condition of self-adjointness, and identifying which shape of the domain reflects into spectrum of prescribed type and position \cite{ColinVerd-1987,Hempel-Seco-Simon-1991-JFA1992,Hempel-Kriecherbauer-Plankensteiner-1997,Behrndt-Khrabustovskyi-2021};
  \item by considering singular self-adjoint perturbations, supported on points or more generally on manifolds of non-zero co-dimension, of an explicit differential operator (Laplace, Dirac, etc.), and identifying which location, shape, strength of the point-like perturbation reflect into spectrum of prescribed type and position
  \cite{albeverio-geyler-1998-CMP2000,Geyler-Pankrashkin_QMat7-1998,Behrndt-Holzmann-2017_JSpecTh2020,Behrndt-Khrabustovskyi-2021,Behrndt-Khrabustovskyi-2021-JFA2022,Behrndt-Holzmann-Stelzer-Stenzel-2022,Benhellal-Pankrashkin-2022}.
 \end{itemize}

 The works cited above provide extra background and older references: in fact, all such mainstream ideally dates back to a very classical result by Kre{\u\i}n \cite[Theorem 23]{Krein-1947} on the existence, for gapped symmetric operators with (equal and) \emph{finite} deficiency indices, of self-adjoint extensions with prescribed finite point spectrum in the gap.

 At the abstract operator-theoretic level it is today known that \emph{any kind of spectrum inside the gap of a gapped symmetric operator $S$ with \emph{infinite} deficiency indices on a separable complex Hilbert space $\cH$ can be generated by a suitable self-adjoint extension of $S$ in $\cH$}. This quite remarkable result is due to Brasche \cite[Theorem 27]{Brasche2001_ContMath2004} as the culmination of a extensive and prolific investigation 
 that includes the notable precursors \cite[Theorem 2.2]{Brasche-Neidhardt-Weidmann-1992}, \cite[Theorem 1]{Brasche-Neidhardt-1996}, \cite[Theorem 6.1]{Albeverio-Brasche-Neidhardt-JFA1998}, \cite[Theorems 4.5 and 6.3]{Brasche-JOT2000} from collaborations with Albeverio, Neidhardt, and Weidmann, and also from hints by Simon \cite{Brasche-JOT2000}.

 Despite the `ultimate' character of such a general result, the focus remains active both on the abstract operator-theoretic mechanisms for the appearance of spectrum of prescribed type and location within the gap (on which further important comments are going to follow later in this Introduction), and on a variety of explicit differential operators of relevance in applications, self-adjointly realised as extensions of gapped symmetric operators.

 In this respect, it is instructive to highlight the recent activity by Behrndt and collaborators, as well as by Pankrashkin and collaborators, concerning the appearance of pre-definite parts of essential spectrum for self-adjoint realisations of Schr\"{o}dinger and Dirac operators with singular perturbation supported on a suitable manifold of positive co-dimension in $\mathbb{R}^d$. They provided valuable insight on such spectral phenomenon.

  In a first significant model, a quantum particle scatters in one dimension across an infinite sequence of delta-like impurities supported at points $x_1,x_2,\dots$. The Hamiltonian is formally of the type
 \[
  H_{\Sigma,\gamma}\;=\;-\frac{\ud^2}{\ud x^2}+\sum_{k=1}^\infty \gamma_k \delta(\cdot-x_k)\,,
 \]
 for a given sequence $\Sigma\equiv(x_k)_{k\in\mathbb{N}}\subset (\ell_-,\ell_+)$ of singularity locations, and a given sequence $\gamma\equiv(\gamma_k)_{k\in\mathbb{N}}\subset \mathbb{R}$ of interaction magnitudes, and it is self-adjointly realised on $L^2(\ell_-,\ell_+)$ by means of the boundary conditions
 \[
  \psi(x_k^+)\;=\;\psi(x_k^-)\,,\qquad \psi'(x_k^+)-\psi'(x_k^-)\;=\;\gamma_k\psi(x_k)\,,\qquad \psi(\ell_\pm)\;=\;0
 \]
 imposed on functions $\psi\in H^1((\ell_-,\ell_+)\cap H^2((\ell_-,\ell_+)\setminus\bigcup_{k\in\mathbb{N}} \{x_k\})$. This is by now a well-established construction, comprehensively discussed, e.g., in  \cite[Chapters II.2 and III.2]{albeverio-solvable} as well as in \cite{Albe-Kostenko-Malamud-JMP2010,Kostenko-Malamud-JDE2010,Behrndt-Khrabustovskyi-2021-JFA2022}). Clearly, any such $H_{\Sigma,\gamma}$ is a self-adjoint extension of the non-negative symmetric operator
 \[
  \mathring{H}\;=\;\Big(-\frac{\ud^2}{\ud x^2}\Big)\bigg|_{C^\infty_c((\ell_-,\ell_+)\setminus\bigcup_{k\in\mathbb{N}} \{x_k\})}\,.
 \]
 It has been recently proved \cite[Theorems 1.1 and 4.4]{Behrndt-Khrabustovskyi-2021-JFA2022} that for bounded $(\ell_-,\ell_+)$ any pre-definite closed and lower semi-bounded subset $S_{\mathrm{ess}}\subset \mathbb{R}$ is the \emph{essential spectrum} of one of the $H_{\Sigma,\gamma}$'s, for a suitable choice of $(x_k)_{k\in\mathbb{N}}$ and $(\gamma_k)_{k\in\mathbb{N}}$. An analogous result \cite[Theorems 3.1 and 3.3]{Behrndt-Khrabustovskyi-2021}, under the constraint $\{0\}\subset S_{\mathrm{ess}}\subset[0,+\infty)$, has been also proved for the counterpart model formally described by
 \[
  \widetilde{H}_{\Sigma,\gamma}\;=\;-\frac{\ud^2}{\ud x^2}+\sum_{k=1}^\infty \gamma_k \langle\delta'_{x_k},\cdot\rangle\,\delta'_{x_k}\,,
 \]
 where $\delta'_{x_k}\equiv\delta'(\cdot-x_k)$ is the distributional derivative of $\delta(\cdot-x_k)$ and $\langle\delta'_{x_k},\cdot\rangle$ denotes its action $\phi\mapsto\langle\delta'_{x_k},\phi\rangle=-\phi'(x_k)$ on test functions $\phi$: the $\widetilde{H}_\gamma$'s constitute another class of self-adjoint extensions of $\mathring{H}$ and are self-adjointly realised with action $-\frac{\ud^2}{\ud x^2}$ and boundary conditions of self-adjointness
  \[
  \psi'(x_k^+)\;=\;\psi'(x_k^-)\,,\qquad \psi(x_k^+)-\psi(x_k^-)\;=\;\gamma_k\psi'(x_k)\,.
 \]
 Either result relies on a similar idea as in the preceding construction \cite{Hempel-Seco-Simon-1991-JFA1992}  of bounded domains of `room-and-passages' type, where the self-adjoint Neumann Laplacian is proved to have also essential spectrum -- somewhat contrarily to the intuition that self-adjoint Laplace-type operators on bounded domains only have purely discrete spectrum (or, equivalently, a compact resolvent).

 A second type of model under deep scrutiny, from the same perspective of emergence of essential spectrum in a gap by self-adjoint extension, concerns the three-dimensional Dirac operator with critical combination of electrostatic and Lorenz scalar shell interaction supported at compact surfaces. In this case the Hamiltonian governing the evolution of the quantum particle is the operator $H_{\Sigma,\varepsilon,\mu}$ in $L^2(\mathbb{R}^3,\mathbb{C}^4)$ formally acting as
 \[
  H_{\Sigma,\varepsilon,\mu}\;=\;-\ii\sum_{j=1}^3 \alpha_j \partial_{x_j} + m \beta + (\varepsilon\mathbbm{1}+\mu\beta)\delta_\Sigma
 \]
 for given mass, electrostatic shell, and Lorenz shell parameters, respectively, $m\geqslant 0$ and $\varepsilon,\mu\in\mathbb{R}$, where $\alpha_1,\alpha_2,\alpha_3,\beta$ are four anti-commuting $4\times 4$ Hermitian matrices squaring to the identity -- for concreteness, the usual choice of the Pauli matrices, $x\equiv(x_1,x_2,x_3)\in\mathbb{R}^3$, and $\delta_\Sigma$ is the delta distribution along the normal to $\Sigma$ at each point of a given sufficiently smooth compact surface $\Sigma\subset\mathbb{R}^3$. Thus, $H_{\Sigma,\varepsilon,\mu}$ is self-adjointly realised by means of $\delta$-type boundary conditions along the normal direction at each point of $\Sigma$, and is by construction a self-adjoint extension of
 \[
  \mathring{H}\;=\;\bigg(\ii\sum_{j=1}^3 \alpha_j \partial_{x_j} + m \beta\bigg)\bigg|_{C^\infty_c(\mathbb{R}^3\setminus\Sigma,\mathbb{C}^4)}\,,
 \]
 or also, from another point of view, $H_{\Sigma,\varepsilon,\mu}$ is a \emph{singular perturbation} (in the sense of, e.g., \cite{albverio-kurasov-2000-sing_pert_diff_ops}) of the self-adjoint operator $H_0$ defined by
 \[
  \mathcal{D}(H_0)\;=\;H^1(\mathbb{R}^3,\mathbb{C}^4)\,,\qquad H_0\psi\;=\;-\ii\sum_{j=1}^3 \alpha_j \partial_{x_j}\psi + m \beta\psi\,.
 \]
 As a matter of fact, $H_0$ is gapped with 
 \[
  \sigma(H_0)\;=\;\sigma_{\mathrm{ess}}(H_0)=(-\infty,-m]\cup[m,+\infty)\,,
 \]
 and for $\varepsilon^2-\mu^2\neq 4$ each singular perturbations $H_{\Sigma,\varepsilon,\mu}$ preserve the essential spectrum and only produces at most finitely many eigenvalues of finite multiplicity within the gap \cite{Behrndt-Exner-Holtzmann-Loto-2019-Diracshell,Behrndt-Holzmann-Stelzer-Stenzel-2022,OurmBonafos-Pizzichillo-INdAM2021}. This is analogous to the more classical problem (see \cite{MG_DiracCoulomb2017,GM-2017-DC-EV} and references therein)  of the self-adjoint realisations of the Dirac-Coulomb Hamiltonian as extensions of
 \[
  \bigg(\ii\sum_{j=1}^3 \alpha_j \partial_{x_j} + m \beta+\frac{\nu}{\,|x|\,}\bigg)\bigg|_{C^\infty_c(\mathbb{R}^3\setminus\{0\},\mathbb{C}^4)}
 \]
 in the relevant regime $\nu\in\mathbb{R}$, $\frac{\sqrt{3}}{2}<|\nu|<1$ of existence of multiple extensions: each such  self-adjoint realisation comes with the same essential spectrum as $\sigma_{\mathrm{ess}}(H_0)$ and with additional point spectrum consisting of an infinity of eigenvalues of finite multiplicity, accumulating to one of the two edges of the gap $(-m,m)$ (depending on the sign of $\nu$). Instead, in the critical regime $\varepsilon^2-\mu^2=4$ it has been recently revealed that $H_{\Sigma,\varepsilon,\mu}$ above may have additional essential spectrum (as compared to $H_0$) in the gap $(-m,m)$: in particular \cite[Theorem 6.6]{Behrndt-Holzmann-Stelzer-Stenzel-2022}, if $\Sigma$ is partially flat, then the point $-\frac{m\mu}{\varepsilon}$ belongs to $\sigma_{\mathrm{ess}}(H_{\Sigma,\varepsilon,\mu})$. Even more strikingly \cite[Theorem 1]{Benhellal-Pankrashkin-2022}, in terms of the constant
 \[
  A_{\Sigma}\;:=\;\max_{x\in\Sigma}|\kappa_1(x)-\kappa_2(x)|\,,
 \]
 where $\kappa_1$ and $\kappa_2$ are the two principal curvatures of the smooth and compact surface $\Sigma$, one has
 \[
  \sigma_{\mathrm{ess}}(H_{\Sigma,\varepsilon,\mu})\;=\;\sigma_{\mathrm{ess}}(H_0)\cup\Big[-\frac{m\mu}{\varepsilon}-\frac{A_\Sigma}{\,2\varepsilon}\,,-\frac{m\mu}{\varepsilon}+\frac{A_\Sigma}{\,2\varepsilon}\,\Big]\,.
 \]
  When $\Sigma$ is a sphere (or a union of disjoint spheres), and only in that case, $\kappa_1\equiv\kappa_2$ \cite[Prop.~4, Sect.~3.2]{DoCarmo_DiffGeom-2016}, thereby implying that for such $\Sigma$ the operator $H_{\Sigma,\varepsilon,\mu}$ produces one single point of additional essential spectrum.

 As a matter of fact, a common trait of the recent activity illustrated above
 is that the emergence of additional essential spectrum in the gap is detected through various types of analyses that are all particularly laborious, to say the least, both in explicit models of singular perturbations of certain Schr\"{o}dinger or Dirac operators, and in abstract settings.

 Thus, for instance, the spectral analysis \cite{Benhellal-Pankrashkin-2022} of the above-mentioned three-dimensional Dirac operator with critical electrostatic-Lorenz scalar shell interaction goes first through a reformulation in terms of matrix singular integral operators over the surface $\Sigma$ and a subsequent reduction, by means of the theory of block operator matrices, to the analysis of the essential spectrum of a particular pencil (in fact, a Schur complement) of integral operators over $\Sigma$. In turn, this analysis is performed in terms of principal symbols of suitable auxiliary operators, and it relies on a subtle toolbox concerning the spectrum of classical pseudo-differential operators of order zero on compact Riemannian manifolds.

 As for the emergence of pre-definite regions of essential spectrum of one-di\-men\-sion\-al singular Schr\"{o}dinger operators with countably infinite point-like $\delta$-in\-ter\-ac\-tions, this is obtained in \cite{Behrndt-Khrabustovskyi-2021-JFA2022} by choosing location and magnitude of the singular interactions so as to partition $(\ell_-,\ell_+)$ into adjacent intervals on which one builds an auxiliary operator of singular point interaction in the very same spirit of the room-and-passage construction \cite{Hempel-Seco-Simon-1991-JFA1992} for Neumann Laplacians: such auxiliary operator is constructed so that its eigenvalues constitute a dense of the prescribed spectral set $S_{\mathrm{ess}}$, and the final $H_{\Sigma,\gamma}$ is then built as a compact perturbation, in the resolvent sense, of the auxiliary operator; it therefore displays the same essential spectrum. (The singular $\delta'$-interaction case is treated along the same ideas \cite{Behrndt-Khrabustovskyi-2021}.)

 The abstract operator-theoretic analysis that culminated with the above-mention\-ed result by Brasche is no cheap business either, as the subtle works \cite{Brasche-Neidhardt-1996,Albeverio-Brasche-Neidhardt-JFA1998,Brasche-JOT2000,Brasche2001_ContMath2004} show. In such investigations, starting with a symmetric operator $S$ with infinite deficiency indices and with gap $J$, the self-adjoint extension of $S$ which provides the pre-definite region and type of spectrum inside the gap is eventually identified `indirectly' up to unitary equivalence, in the sense that for any arbitrary auxiliary self-adjoint operator $A_{\mathrm{aux}}$ an operator $A=A^*\supset S$ is shown to exists such that $A E_A(J)$ is unitarily equivalent to $A_{\mathrm{aux}} E_{A_{\mathrm{aux}}}(J)$ (the notation $E_{H}(\cdot)$ standing for the operator-valued measure associated with the self-adjoint $H$). This line of reasoning, as far as the existence of extensions with pre-definite absolutely continuous or singular continuous spectrum in the gap is concerned, employs ingenious arguments from measure theory and block operators, and in a sense eventually leaves the `explicit mechanism' of construction of $A$ from $S$ somewhat obfuscated.
 
 It is instead cleaner, in its basic mechanisms, how to abstractly construct self-adjoint extensions with only prescribed \emph{point spectrum} within the gap -- hence also with prescribed \emph{essential spectrum}, as one may select a countable dense in the assigned region of new essential spectrum and build the desired self-adjoint extension so that it admits the previously selected points as eigenvalues. The idea comes from the already mentioned `starting' result \cite[Theorem 23]{Krein-1947} by Kre{\u\i}n for finite deficiency indices: in case of infinite deficiency indices the same idea was refined and adapted in \cite[Theorem 2.2]{Brasche-Neidhardt-Weidmann-1992}, \cite[Lemma 2.1 and Proposition 3.1]{Albeverio-Brasche-Neidhardt-JFA1998}, \cite[Lemma 29 and Corollary 30]{Brasche2001_ContMath2004}, \cite[Theorem 3.3]{Albe-Dudkin-Konst-Koshm-2005}. This approach will be further discussed in Sections \ref{sec:preparation-unital} and \ref{sec:geninfind} in comparison to the construction presented in the present work.

%
%

 In fact, it is precisely \emph{the construction, in the infinite deficiency index case, of self-adjoint extensions with pre-definite point spectrum (respectively, essential spectrum) within the gap} which is at the core of this note, both in view of the abstract results cited here above, and of the previously reviewed recent analysis of explicit models. 
 
 In particular, \emph{the perspective here is to identify the mechanism of emergence of the desired extension with respect to general self-adjoint extension schemes}. 
 
 That is: how does the prescribed set of point spectrum, or of essential spectrum, select the desired operator, out of the general family of self-adjoint extensions produced and classified by abstract extension schemes?

%
%
%
%
%
%

 Recently, Behrndt and Khrabustovskyi \cite[Theorem 4.3]{Behrndt-Khrabustovskyi-2021} have partially explored this direction, producing an elegant, clean and short proof of Brasche's theorem in the case of new essential spectrum, which has the value of characterising the desired extension as a convenient restriction of $S^*$ with respect to the decomposition
 \[
  \mathcal{D}(S^*)\;=\;\mathcal{D}(\overline{S})\dotplus\ker(S^*-\mu\mathbbm{1})
 \]
 for $\mu$ in the gap of $S$, the latter identity being the starting point of any abstract self-adjoint extension scheme (see \eqref{eq:domainadjoint} below). Their argument, however, relies on the crucial additional assumption that $S$ admit a self-adjoint extension with \emph{compact resolvent}: this condition is exploited in multiple steps, in particular showing that from an arbitrary self-adjoint operator $\Xi$ on the infinite-dimensional Hilbert space $\ker(S^*-\mu\mathbbm{1})$ one can build a self-adjoint restriction $A$ of $S^*$ which is, upon identifying $\Xi$ as an operator on the whole $\cH=\ker(S^*-\mu\mathbbm{1})\oplus\mathrm{ran}(\overline{S}-\mu\mathbbm{1})$, a compact perturbation in the resolvent sense of $\Xi$ itself, thereby implying that $\sigma_{\mathrm{ess}}(A)=\sigma_{\mathrm{ess}}(\Xi)$.

 In this note, the emergence of arbitrary new point spectrum and essential spectrum in the gap when the deficiency index is infinite is shown in its simple and clean operator-theoretic mechanism within the general scheme of self-adjoint extensions, with no restriction of separability, compactness of the resolvent, and the like.

 Useful preparations are discussed in Section \ref{sec:preparation-unital}, including a crucial construction in the prototypical case of \emph{unital} deficiency index.

 The general mechanism when the deficiency index is \emph{infinite} is then discussed in Section \ref{sec:geninfind}.

 \section{Preparation: unital deficiency index}\label{sec:preparation-unital}

 Let us start by recalling the following general construction and classification of \emph{all} the self-adjoint extensions of a given gapped symmetric operator. Up to an additive shift by a suitable multiple of the identity, it is clearly non-restrictive to assume the gap to be around zero, which will be done throughout.

 This is a classical result that emerged about seven decades ago as the culmination of seminal works by Kre{\u\i}n \cite{Krein-1947}, Vi\v{s}ik \cite{Vishik-1952}, and Birman \cite{Birman-1956}, and is therefore fair to refer to as the Kre{\u\i}n-Vi\v{s}ik-Birman theory, or also Kre{\u\i}n-Vi\v{s}ik-Birman-Grubb theory in view of the subsequent adaptation by Grubb \cite{Grubb-1968} to the more general, but technically similar problem of the closed extensions of closed operators (see, e.g., \cite{GMO-KVB2017} and \cite[Chapter 2]{GM-SelfAdj_book-2022} and the references therein, as well as \cite[Chapter 13]{Grubb-DistributionsAndOperators-2009} and \cite{KM-2015-Birman}). In more recent times it has become customary to re-derive Theorem \ref{thm:VB-representaton-theorem_Tversion2} below within the theory of boundary triplets (one updated reference to which is \cite{Behrndt-Hassi-deSnoo-boundaryBook}), a modern reformulation of the approach by Vi\v{s}ik and Birman. The original classification a la Birman is particularly efficient and informative in the mathematical analysis of several quantum mechanical models of current interest -- see, e.g., \cite{MO-2016,MO-2017,MG_DiracCoulomb2017,GM-2017-DC-EV,GM-hydrogenoid-2018,GMP-Grushin2-2020,GM-Grushin3-2020,M2020-BosonicTrimerZeroRange}.

 \begin{theorem}\label{thm:VB-representaton-theorem_Tversion2}\index{theorem!Kre{\u\i}n-Vi\v{s}ik-Birman (self-adjoint extension)}
 Let $S$ be a densely defined symmetric operator in a complex Hilbert space $\cH$, which admits a self-adjoint extension $S_\mathrm{D}$ that has everywhere defined bounded inverse on $\cH$ -- equivalently, assume that $S$ is a densely defined gapped symmetric operator with zero in the gap. There is a one-to-one correspondence between the  family of the self-adjoint extensions of  $S$ in $\cH$ and the family of the self-adjoint operators in Hilbert subspaces of $\ker S^*$. If $T$ is any such operator, in the correspondence $T\leftrightarrow S_T$ each self-adjoint extension $S_T$ of $S$ is given by
\begin{equation}\label{eq:ST-2}
\begin{split}
S_T\;&:=\;S^*\upharpoonright\mathcal{D}(S_T)\,, \\
\mathcal{D}(S_T)\;&:=\;\left\{f+S_\mathrm{D}^{-1}(Tv+w)+v\left|
\begin{array}{c}
f\in\mathcal{D}(\overline{S})\,,\;v\in\mathcal{D}(T) \\
w\in\ker S^*\cap\mathcal{D}(T)^\perp
\end{array}
\right.\right\}.
\end{split}
\end{equation}
\end{theorem}

 Observe that indeed $\mathcal{D}(S)\subset\mathcal{D}(S_T)\subset\mathcal{D}(S^*)$, as
 \begin{equation}\label{eq:domainadjoint}
  \mathcal{D}(S^*)\;=\;\mathcal{D}(\overline{S})\dotplus S_{\mathrm{D}}^{-1}\ker S^*\dotplus \ker S^*
 \end{equation}
 (see, e.g., \cite[Proposition 2.2]{GM-SelfAdj_book-2022}).

 It is convenient and customary to refer to the operator $T$ indexing the self-adjoint extension $S_T$ as the \emph{(Birman) extension parameter} of $S_T$.

%
%
 It is known that quite an amount of information on $S_T$ may be read out from the often `easier' labelling operator $T$: one, in particular, is invertibility (see, e.g., \cite[Theorem 2.22]{GM-SelfAdj_book-2022}):	

  \begin{theorem}\label{thm:invertibility-gen}
  Under the assumptions Theorem \ref{thm:VB-representaton-theorem_Tversion2}:
\begin{enumerate}
 \item[(i)] $\ker S_T=\ker T$, and therefore $S_T$ is injective $\Leftrightarrow$ $T$ is injective;
 \item[(ii)] $S_T$ is surjective $\Leftrightarrow$ $T$ is surjective;
  \item[(iii)] $S_T$ is invertible in the whole $\cH$ $\Leftrightarrow$ $T$ is invertible in the whole $\overline{\mathcal{D}(T)}$.
\end{enumerate}
\end{theorem}
 
%
%
%

 Now, as a first application of the general results above, consider a gapped symmetric operator around zero with \emph{unital deficiency index}, i.e.,
 \begin{equation}\label{eq:A1}
  \begin{array}{l}
   \textrm{$S$ is a densely defined and symmetric operator in $\cH$} \\
   \textrm{with a gap $(a,b)\ni 0$ in the sense of (1.1),} \\
   \textrm{\emph{equivalently},} \\
   \textrm{with distinguished $S_{\mathrm{D}}=S_{\mathrm{D}}^*\supset S$ such that $S_{\mathrm{D}}^{-1}\in\mathcal{B}(\cH)$.}
  \end{array}
 \end{equation}
 and 
 \begin{equation}\label{eq:Aunital}
  d(S)\,=\,1\,,
 \end{equation}
 where, under assumption \eqref{eq:A1},
 \begin{equation}\label{eq:def-def-ind}
  d(S)\;:=\;\dim\ker S^*\;=\;\dim\ker(S^*-z\mathbbm{1})\qquad \forall z\in\mathbb{C}\setminus\mathbb{R}\,.
 \end{equation}
 In this context it is meant that the gap $(a,b)$ may be finite (namely, finite $a$ and $b$) or semi-infinite (say, $a=-\infty$ and $b>0$ finite). Recall that the fact that the gap condition \eqref{eq:A1} ensures the validity of the identity in \eqref{eq:def-def-ind} for any complex non-real $z$ makes it legitimate to refer to $d(S)$ as \emph{the} deficiency index of $S$.

 Observe also that \eqref{eq:A1} implies $0\notin\sigma_{\mathrm{point}}(\overline{S})$, i.e., $\ker\overline{S}=\{0\}$: $\overline{S}$ is thus invertible on its range.

 An instructive special case of \eqref{eq:A1}, ubiquitous in applications, is that of a lower semi-bounded and densely defined $S$, meaning 
 \begin{equation}\label{eq:mS}
 \mathfrak{m}(S)\;:=\;\inf_{\substack{\psi\in\mathcal{D}(S) \\ \psi\neq 0}}\frac{\langle\psi,S\psi\rangle}{\;\|\psi\|^2}\;>\;-\infty\,,
\end{equation}
 with strictly positive lower bound, that is, $\mathfrak{m}(S)>0$.


 For generic $\lambda\in(a,b)$, set 
 \begin{equation}\label{eq:defSlambda}
  \mathcal{D}_\lambda\,:=\,\mathcal{D}(\overline{S})\dotplus\ker(S^*-\lambda\mathbbm{1})\,,\qquad S_\lambda\,:=\,S^*\upharpoonright\mathcal{D}_\lambda\,.
 \end{equation}

  Observe that $\mathcal{D}(S^*)=\mathcal{D}(S^*-\lambda\mathbbm{1})=\mathcal{D}((S-\lambda\mathbbm{1})^*)$ ($\lambda$ being real), and moreover having taken $\lambda$ in the gap guarantees that $S-\lambda\mathbbm{1}$ too has a distinguished self-adjoint extension $(S-\lambda\mathbbm{1})_{\mathrm{D}}$ which has everywhere defined and bounded inverse on $\cH$. Owing to \eqref{eq:domainadjoint} above one has
    \begin{equation}\label{eq:adjdomainsplitting}
   \mathcal{D}(S^*-\lambda\mathbbm{1})\,=\,\mathcal{D}(\overline{S}-\lambda\mathbbm{1})\dotplus (S-\lambda\mathbbm{1})_{\mathrm{D}}^{-1}\ker(S^*-\lambda\mathbbm{1})\dotplus\ker(S^*-\lambda\mathbbm{1})\,.
  \end{equation}
   Thus, since $\mathcal{D}(\overline{S})=\mathcal{D}(\overline{S}-\lambda\mathbbm{1})$, it is indeed confirmed that the sum in \eqref{eq:defSlambda} is direct.

   Actually, \eqref{eq:defSlambda} is a way of introducing the crucial operator $S_\lambda$ that is \emph{not} the convenient one needed for exporting our argument to the case of infinite deficiency case, and for this reason it will be replaced in a moment by a more informative, equivalent characterisation (Lemma \ref{lem:SlambdaSbeta} below). We simply started with \eqref{eq:defSlambda} for its direct connection with several analogous operators exploited in past analyses from the above-mentioned literature.

    \begin{lemma}\label{lem:Slambdaunitalsa}
  $S_\lambda$ is a self-adjoint extension of $S$.
 \end{lemma}

 \begin{proof}
  By construction, $S_\lambda$ is a densely defined extension of $S$. In fact, $S_\lambda$ is self-adjoint; for,
   \begin{equation}\label{eq:Slambdamp}
    S_\lambda\,=\,S^*\upharpoonright\mathcal{D}_\lambda\,=\,(S^*-\lambda\mathbbm{1})\upharpoonright\mathcal{D}_\lambda+\lambda\mathbbm{1}
   \end{equation}
   and $(S^*-\lambda\mathbbm{1})\upharpoonright\mathcal{D}_\lambda$ is a self-adjoint extension of $S-\lambda\mathbbm{1}$. The latter claim is checked by means of Theorem \ref{thm:VB-representaton-theorem_Tversion2} applied to $S-\lambda\mathbbm{1}$, choosing now the labelling operator $T$ to be the zero operator on the whole $\ker(S^*-\lambda\mathbbm{1})$.
 \end{proof}
   
   \begin{remark}
    When in particular the operator $S$ from \eqref{eq:A1} satisfies $\mathfrak{m}(S)>0$, the above construction for $S_\lambda$ requires $\lambda<\mathfrak{m}(S)$. In this case, $\mathfrak{m}(S-\lambda\mathbbm{1})=\mathfrak{m}(S)-\lambda>0$, and the distinguished extension in \eqref{eq:adjdomainsplitting} can be taken to be the \emph{Friedrichs extension} $(S-\lambda\mathbbm{1})_{\mathrm{F}}$ of $S-\lambda\mathbbm{1}$. With this choice, $(S^*-\lambda\mathbbm{1})\upharpoonright\mathcal{D}_\lambda$ is nothing but the \emph{Kre{\u\i}n-von Neumann extension} $(S-\lambda\mathbbm{1})_{\mathrm{N}}$ of $S-\lambda\mathbbm{1}$ (see, e.g., \cite[Theorem 2.10]{GM-SelfAdj_book-2022}), namely its unique smallest positive self-adjoint extension. 
   Observe that, whereas $(S-\lambda\mathbbm{1})_{\mathrm{F}}=S_{\mathrm{F}}-\lambda\mathbbm{1}$, in general $(S-\lambda\mathbbm{1})_{\mathrm{N}}\neq S_{\mathrm{N}}-\lambda\mathbbm{1}$.
   \end{remark}

   Based on the latter remark, even when going back to the general assumptions \eqref{eq:A1}-\eqref{eq:Aunital} it is informative to refer to $S_\lambda$ as the sum of $\lambda\mathbbm{1}$ plus the `extension of
   Kre{\u\i}n-von Neumann type' of $S-\lambda\mathbbm{1}$ -- see \eqref{eq:Slambdamp} above.

%
%
   
   \begin{lemma}\label{lem:Slambdaev}
    Under the conditions \eqref{eq:A1}-\eqref{eq:Aunital} the gapped operator $S$ admits a self-adjoint extension with eigenvalue given by an arbitrary $\lambda$ in the gap.
   \end{lemma}

   \begin{proof}
    For any such $\lambda$, one such operator is precisely $S_\lambda$. It is indeed a self-adjoint extension of $S$ (Lemma \ref{lem:Slambdaunitalsa}) and moreover \eqref{eq:Slambdamp} shows that on any non-zero $u\in\ker(S^*-\lambda\mathbbm{1})\subset\mathcal{D}(S_\lambda)$ (such $u$ certainly existing, owing to \eqref{eq:Aunital}) one has $S_\lambda u = \lambda u$.    
   \end{proof}

   For arbitrary $\lambda$ in the spectral gap of $S$, $S_\lambda$ is a self-adjoint extension of $S$ with eigenvalue $\lambda$. As such, it \emph{must} correspond to an extension parameter $T$ in the sense of the general classification of Theorem \ref{thm:VB-representaton-theorem_Tversion2}. Since $S$ has by assumption unital deficiency index, $T$ is in this case the self-adjoint operator acting in the one-dimensional space $\ker(S^*-\lambda\mathbbm{1})$ as the multiplication by a real number $\beta$. Thus, in the notation therein, $S_\lambda$ must have the form $S_T$ for some $T$ of that sort, for which it is convenient to use the ad hoc notation $S^{[\beta]}$ (meaning: the extension $S_T$ where $T$ is the multiplication by $\beta$). Explicitly, specialising Theorem \ref{thm:VB-representaton-theorem_Tversion2} for unital deficiency index,
   \begin{equation}\label{eq:domainSbeta}
    \mathcal{D}(S^{[\beta]})\;=\;\;\left\{f+S_\mathrm{D}^{-1}(\beta u)+u\left|
\begin{array}{c}
f\in\mathcal{D}(\overline{S})\,, \\
u\in\ker S^*
\end{array}\!\!
\right.\right\}
   \end{equation}
   The choice $\beta=\infty$ is tacitly included, and corresponds to the Friedrichs extension $S_\mathrm{F}$ of $S$.

   \begin{proposition}\label{lem:SlambdaSbeta}
    $S_\lambda=S^{[\beta]}$ for
    \begin{equation}\label{eq:beta-lambda-1}
     \beta\,=\,\lambda\,\big\langle w_0,S_{\mathrm{D}}(S_{\mathrm{D}}-\lambda\mathbbm{1})^{-1}w_0\big\rangle\,,
    \end{equation}
    where $w_0$ is a unit vector (i.e., $\|w_0\|=1$) spanning $\ker S^*$ (formula \eqref{eq:beta-lambda-1} is clearly insensitive of the phase of $w_0$).
   \end{proposition}

   \begin{proof}
   As argued above, for given $S$ and $\lambda$ (and, a priori, for given $S_{\mathrm{D}}$ and $(S-\lambda\mathbbm{1})_{\mathrm{D}}$ as well), there exists one and only one $\beta\in\mathbb{R}\cup\{\infty\}$, depending on such data, for which $S_\lambda=S^{[\beta]}$, or, equivalently, $\mathcal{D}_\lambda=\mathcal{D}(S^{[\beta]})$. 
   This means, in view of \eqref{eq:domainSbeta}, that for arbitrary $\phi\in\mathcal{D}(\overline{S})$ and $v\in\ker(S^*-\lambda\mathbbm{1})$, the element $\phi+v\in\mathcal{D}_\lambda$ can be re-written as
   \[\tag{i}\label{lemma-i}
    \phi+v\;=\;f+S_{\mathrm{D}}^{-1} (\beta u) + u
   \]
   for uniquely determined $f\equiv f_{\phi,v,\lambda}\in\mathcal{D}(\overline{S})$ and $u\equiv u_{\phi,v,\lambda}\in\ker S^*$.

  Now, decompose
  \[\tag{ii}\label{lemma-ii}
   v\,=\,\overline{S}x+z
  \]
 with $x\in\mathcal{D}(\overline{S})$ and $z\in\ker S^*$, uniquely determined via the orthogonal direct sum $\cH=\,\ran\,\overline{S}\oplus\ker S^*$ (recall that under the current assumptions, $\overline{\mathrm{ran}\,S}=\ran\,\overline{S}$).

  Applying $S^*$ on both sides of \eqref{lemma-i} yields $\overline{S}\phi+\lambda \overline{S}x+\lambda z=\overline{S}f+\beta u$, whence
  \[
   \overline{S}\phi+\lambda \overline{S}x-\overline{S}f\,=\,\beta u-\lambda z\,.
  \]
  As the l.h.s.~above belongs to $\mathrm{ran}\,{\overline{S}}$ and the r.h.s.~belongs to $\ker S^*$, necessarily both sides vanish and
  \[\tag{iii}\label{lemma-iii}
   f\,=\,\phi+\lambda x\,,\qquad \beta u\,=\,\lambda z\,.
  \]
 By means of \eqref{lemma-iii} above, \eqref{lemma-i} is re-written as
 \[
  \begin{split}
   \phi+v\,&=\,\phi+\lambda x+S_{\mathrm{D}}^{-1} (\lambda z) + u \\
   &=\,\phi+\lambda S_{\mathrm{D}}^{-1} (\overline{S}x+z)+u \\
   &=\,\phi+\lambda S_{\mathrm{D}}^{-1}v+u\,,
  \end{split}
 \]
  yielding
    \[\tag{iv}\label{lemma-iv}
   u\,=\,(\mathbbm{1}-\lambda S_{\mathrm{D}}^{-1})v\,=\,(S_{\mathrm{D}}-\lambda\mathbbm{1})S_{\mathrm{D}}^{-1}v\,.
  \]
  
  Formulas \eqref{lemma-iii} and \eqref{lemma-iv} provide the expression for $f\in\mathcal{D}(\overline{S})$ and $u\in\ker S^*$, for given $\phi+v\in\mathcal{D}_\lambda$, satisfying \eqref{lemma-i}. In order to find the dependence of $\beta$ on $\lambda$ it is convenient to combine \eqref{lemma-iv} with the second of \eqref{lemma-iii}: this gives
  \[
   \lambda z\,=\,\beta u\,=\,\beta (S_{\mathrm{D}}-\lambda\mathbbm{1})S_{\mathrm{D}}^{-1}v\,,
  \]
  whence (using also \eqref{lemma-ii})
  \[
   \lambda (S_{\mathrm{D}}-\lambda\mathbbm{1})^{-1}z\,=\,\beta S_{\mathrm{D}}^{-1}v\,=\,\beta S_{\mathrm{D}}^{-1}(\overline{S}x+z)\,.
  \]
  Recall, indeed, that since $\lambda$ belongs to the gap of $S$ then $S_{\mathrm{D}}-\lambda\mathbbm{1}$ has everywhere defined and bounded inverse in $\cH$.
  Applying $S^*=(S^*-\lambda\mathbbm{1})+\lambda\mathbbm{1}$ to the latter identity yields
  \[
   \lambda z +\lambda^2(S_{\mathrm{D}}-\lambda\mathbbm{1})^{-1}z\,=\,\beta \overline{S}x+\beta z\,.
  \]
  Taking the scalar product of each side with $z$ itself gives
  \[
   (\lambda-\beta)\|z\|^2+\lambda^2\big\langle z, (S_{\mathrm{D}}-\lambda\mathbbm{1})^{-1}z \big\rangle\,=\,0\,,
  \]
  as $z\perp \overline{S}x$, whence (in the non-trivial case $z\neq 0$)
     \[\tag{v}\label{lemma-v}
   \beta\,=\,\lambda + \lambda^2\,\frac{\,\big\langle z, (S_{\mathrm{D}}-\lambda\mathbbm{1})^{-1}z \big\rangle\,}{\|z\|^2}\,=\,\lambda\,\frac{\,\big\langle z, S_{\mathrm{D}}(S_{\mathrm{D}}-\lambda\mathbbm{1})^{-1}z \big\rangle\,}{\|z\|^2}\,.
  \]
  
  Clearly, \eqref{lemma-v} above is the desired formula \eqref{eq:beta-lambda-1}.
   \end{proof}

   \begin{remark}\label{1-remark}~
   \begin{enumerate}
    \item $S_\lambda$ is non-invertible when $\lambda=0$, as it admits the eigenvalue zero. Through the identification $S_\lambda=S^{[\lambda]}$ (Proposition \ref{lem:SlambdaSbeta}) this is consistent with the fact (Theorem \ref{thm:invertibility-gen}) that $S^{[\beta]}$ loses its invertibility precisely when $\beta=0$.
    \item In Proposition \ref{lem:SlambdaSbeta}'s proof above, the implication $\lambda=0\Rightarrow\beta=0$ is already deduced from conditions \eqref{lemma-i}-\eqref{lemma-iii}. Indeed, if $\lambda=0$, then $\beta u=\lambda z=0$ 
    irrespective of the choice of $\phi+v\in\mathcal{D}_\lambda$; the option $\beta\neq 0$ would imply that all the $u\equiv u_{\phi,v,\lambda}$'s must vanish, whereby $\phi+v=f+S_{\mathrm{D}}^{-1} (\beta u) + u=f\in\mathcal{D}(\overline{S})$, i.e., $\mathcal{D}_\lambda\subset\mathcal{D}(\overline{S})$, a contradiction.
   \end{enumerate}
   \end{remark}

   \begin{example}
    It is instructive to visualise the main identities of Proposition \ref{lem:SlambdaSbeta}'s proof with a concrete example. In $\cH:=L^2(\mathbb{R}^+)$ consider
    \[
     \mathcal{D}(S)\,:=\,C^\infty_c(\mathbb{R}^+)\,,\qquad Sf\,:=\,-f''+f\,.
    \]
   Such $S$ is a densely defined and lower semi-bounded symmetric operator, with $\mathfrak{m}(S)=1$, thus with gap $(-\infty,1)$. $S$ then admits the Friedrichs extension, $S_{\mathrm{F}}$, which is the distinguished extension $S_{\mathrm{D}}$ for this example. It is standard to argue (see, e.g., \cite[Section 7.1]{GMO-KVB2017}) that $S$ has unital deficiency index and
   \[
    \begin{split}
     \mathcal{D}(\overline{S})\,&=\,H^2_0(\mathbb{R}^+)\,, \\
     \mathcal{D}(S_{\mathrm{F}})\,&=\,H^2(\mathbb{R}^+)\cap H^1_0(\mathbb{R}^+)\,, \\
     \mathcal{D}(S^*)\,&=\,H^2(\mathbb{R}^+)\,, \\
     \ker S^*\,&=\,\mathrm{span}\{e^{-t}\}\,,
    \end{split}
   \]
   all operators $\overline{S}$, $S_{\mathrm{F}}$, $S^*$ acting as $f\mapsto -f''+f$. Thus, the boundary condition for any $f\in\mathcal{D}(\overline{S})$ is $f(0)=f'(0)=0$, and the boundary condition for any $f\in\mathcal{D}(S_{\mathrm{F}})$ is $f(0)=0$. For $\lambda$ in the gap of $S$, namely $\lambda<1$,
   \[
    \ker(S^*-\lambda\mathbbm{1})\,=\,\mathrm{span}\big\{e^{-t\sqrt{1-\lambda}}\big\}\,.
   \]
  Consider $v=p \,e^{-t\sqrt{1-\lambda}}\in \ker(S^*-\lambda\mathbbm{1})$ for arbitrary $p\in\mathbb{C}$. In the decomposition 
  \[
   v\,=\,\overline{S}x+z\,, \qquad x\in\mathcal{D}(\overline{S})\,,\;z\in\ker S^*\,,
  \]
  necessarily $z=q\,e^{-t}$ for some $q\in\mathbb{C}$, and the identity $\langle z,v\rangle_{L^2}=\|z\|^2_{L^2}$ yields
  \[
   q\,=\,p\,\frac{2}{\,\sqrt{1-\lambda}+1\,}\,.
  \]
   Let us now identify the element $u=(S_{\mathrm{F}}-\lambda\mathbbm{1})S_{\mathrm{F}}^{-1}v\in\ker S^*$. With  $w:=S_{\mathrm{F}}^{-1}v$, one writes $u=(S_{\mathrm{F}}-\lambda\mathbbm{1})w$. Here $w$ is determined by the ODE $v=S^*w=-w''+w$ with boundary condition $w(0)=0$: easy computations show that the only $L^2$-solution to this problem is $w=-\frac{p}{\lambda}(e^{-t\sqrt{1-\lambda}}-e^{-t})$. In turn,
   \[
    u\,=\,(S_{\mathrm{F}}-\lambda\mathbbm{1})w\,=\,-w''+(1-\lambda)w\,=\,p\,e^{-t}\,.
   \]
   Since one must have $\beta u=\lambda z$, the above findings for $u$ and $z$ yield
   \[
    \beta\,p\,e^{-t}\,=\,\lambda\,q\,e^{-t}\,,
   \]
   whence finally
   \[
    \beta\,=\,\lambda\,\frac{q}{p}\,=\,\frac{2\lambda}{\,\sqrt{1-\lambda}+1\,}\,.
   \]
   This is the actual dependence $\beta\equiv\beta(\lambda)$. Last, let us show that such expression is precisely the one given by \eqref{eq:beta-lambda-1}. It is non-restrictive to assume $\lambda\neq 0$ henceforth, as the implication $\lambda=0\Rightarrow\beta=0$ is already guaranteed by general arguments (Remark \ref{1-remark}(ii)).
   One has 
   \[
   \begin{split}
    \|z\|_{L^2}^2\,&=\,\frac{\:|q|^2}{2}\,, \\
        \big\langle z, S_{\mathrm{F}}(S_{\mathrm{F}}-\lambda\mathbbm{1})^{-1}z \big\rangle_{L^2}\,&=\,|q|^2\langle e^{-t},S_{\mathrm{F}}(S_{\mathrm{F}}-\lambda\mathbbm{1})^{-1}e^{-t}\rangle_{L^2}\,.
   \end{split}
   \]
%
   The element $h:=(S_{\mathrm{F}}-\lambda\mathbbm{1})^{-1}e^{-t}\in\mathcal{D}(S_{\mathrm{F}})$ is the $L^2$-solution to the problem $-h''+(1-\lambda)h=e^{-t}$, $h(0)=0$, which is easily found to be $h=\frac{1}{\lambda}(e^{-t\sqrt{1-\lambda}}-e^{-t})$. Then, $S_{\mathrm{F}}h=e^{-t\sqrt{1-\lambda}}$, and 
   \[
    \begin{split}
     \beta\,&=\,\lambda\,\frac{\,\big\langle z, S_{\mathrm{F}}(S_{\mathrm{F}}-\lambda\mathbbm{1})^{-1}z \big\rangle_{L^2}\,}{\;\|z\|_{L^2}^2} \\
     &=\,2\lambda\,\langle e^{-t}, S_{\mathrm{F}}h \rangle_{L^2}\,=\,2\lambda\,\langle e^{-t},e^{-t\sqrt{1-\lambda}}\rangle_{L^2} \,=\,\frac{2\lambda}{\,\sqrt{1-\lambda}+1\,}\,,
    \end{split}
   \]
%
   as found before. It is worth observing that $\lambda\mapsto\beta(\lambda)$ is strictly monotone increasing, with $\beta(\lambda)\xrightarrow{\;\lambda\to-\infty\;}-\infty$ and $\beta(\lambda)\xrightarrow{\;\lambda\to 1^-\;}+\infty$. Thus, the collection of the $S_\lambda$'s exhaust the whole family of self-adjoint extensions of $S$, as $\lambda$ runs in the gap of $S$.
   \end{example}

   The above line of reasoning places under the scope of the general classification scheme of Theorem \ref{thm:VB-representaton-theorem_Tversion2} an operator, the above $S_\lambda$, whose construction already appeared, with ad hoc or implicit arguments, in \cite{Brasche-Neidhardt-Weidmann-1992,Albeverio-Brasche-Neidhardt-JFA1998,Brasche2001_ContMath2004,Albe-Dudkin-Konst-Koshm-2007}.

   This is instructive per se, and also paves the way for the analysis of the actual case of interest, a gapped operator $S$ with \emph{infinite} deficiency index, which is the object of the next Section.

   In fact, there is already one distinguished case with infinite deficiency index, crucially present in many applications (a recent topical example is \cite{GMP-Grushin2-2020,GM-Grushin3-2020}), which can be treated by exploiting the preparatory arguments presented so far: symmetric operators defined by an infinite orthogonal sum of gapped operators with unital deficiency index.

   \begin{theorem}\label{thm:infdirsum}
    Let $\cH=\bigoplus_{n\in\mathbb{N}}\cH_n$ be a (countably infinite) orthogonal direct sum of Hilbert spaces, and let $S=\bigoplus_{n\in\mathbb{N}}S_n$ be the operator acting in $\cH$ given by the orthogonal sum of operators $S_n$'s, each of which is a densely defined, gapped, symmetric operator acting in the corresponding Hilbert space $\cH_n$ and having unital deficiency index. Assume furthermore that all gaps have a common intersection $(a,b)\subset\mathbb{R}$ which, without loss of generality, is assumed to contain $0$. Then:
    \begin{enumerate}
     \item[(i)] $S$ is a densely defined, gapped, symmetric operator with infinite deficiency index and admitting self-adjoint extensions;
     \item[(ii)] for any finite or countably infinite collection $(\lambda_k)_k$ in $(a,b)$ there is a self-adjoint extension of $S$ having the $\lambda_k$'s as eigenvalues; if $\lambda$ is any of such values, then its multiplicity as an eigenvalue equals the number of the $\lambda_k$'s being equal to $\lambda$;
     \item[(iii)] for any arbitrary closed subset $K\subset[a,b]$ there is a self-adjoint extension of $S$ whose essential spectrum contains $K$.
    \end{enumerate}
   \end{theorem}

   Here $[a,b]$ denotes the closure of $(a,b)$ also in the case when $(a,b)$ is semi-infinite.

   \begin{proof}[Proof of Theorem \ref{thm:infdirsum}]
    Part (i) is obvious, and part (iii) is a direct consequence of part (ii), by choosing the collection $(\lambda_k)_k\subset K\cap(a,b)$ to be a dense of $K$.
    
    Concerning part (ii), consider non-restrictively the \emph{countable} collection $(\lambda_n)_{n\in\mathbb{N}}$ in $(a,b)$: then, one example of the claimed self-adjoint extension is
    \[
     \widetilde{S}\,=\,\bigoplus_{n\in\mathbb{N}}\, S_{n,\lambda_n}\,,
    \]
    where each $S_{n,\lambda_n}$ is the self-adjoint extension of $S_n$ in $\cH_n$ defined by
    \[
      \mathcal{D}(S_{n,\lambda_n})\,:=\,\mathcal{D}(\overline{S_n})\dotplus\ker(S_n^*-\lambda\mathbbm{1})\,,\qquad S_{n,\lambda_n}\,:=\,S_n^*\upharpoonright\mathcal{D}(S_{n,\lambda_n})\,.
    \]
    In other words, $S_{n,\lambda_n}$ is for $S_n$ the analogue of the operator $S_\lambda$ defined in \eqref{eq:defSlambda} with respect to $S$ therein. The fact that $S_{n,\lambda_n}$ extends self-adjointly $S_n$ in $\cH_n$ follows from Lemma \ref{lem:Slambdaunitalsa}. Thus, $\widetilde{S}$ is indeed a self-adjoint extension of $S$ in $\cH$. Furthermore, as seen in the proof of Lemma \ref{lem:Slambdaev}, given a (certainly existing) $v_n\in\ker(S_n^*-\lambda_n\mathbbm{1})\setminus\{0\}\subset\mathcal{D}(S_{n,\lambda_n})\subset\mathcal{D}(S^*)$, one has $\widetilde{S}v_n=S_{n,\lambda_n}v_n=\lambda_n v_n$. This shows that all the $\lambda_n$'s are eigenvalues of $\widetilde{S}$. The conclusion on the multiplicity of such eigenvalues is then straightforward.       
   \end{proof}

   For completeness of reference, infinite orthogonal sums of operators with common gap, each with strictly positive deficiency index, were already analysed in \cite[Theorems 6.1 and 6.2]{Albeverio-Brasche-Neidhardt-JFA1998} and \cite{Brasche-Malamud-Neidhardt-1999-2000} in the context of self-adjoint extensions producing prescribed spectrum in the gap. This was turn based on the the operator-theoretic toolbox devised in \cite[Theorem 2.2]{Brasche-Neidhardt-Weidmann-1992} and \cite[Lemma 2.1 and Proposition 3.1]{Albeverio-Brasche-Neidhardt-JFA1998}, which is \emph{not} explicitly modelled on the extension scheme revisited here in Theorem \ref{thm:VB-representaton-theorem_Tversion2}.

   \section{General case of gapped operators with infinite deficiency index}\label{sec:geninfind}

   Let us finally consider in this Section the general and actual case of interest, namely the creation by self-adjoint extension of arbitrary (point and) essential spectrum within the gap of a densely defined symmetric operator with infinite deficiency index.

     This means that from now on we consider an \emph{infinite-dimensional} Hilbert space $\cH$ and, acting in it,  a densely defined symmetric operator $S$ satisfying the general gap condition \eqref{eq:A1} supplemented with the additional assumption
     \begin{equation}\label{eq:Ass-inf-def-ind}
      d(S)\,=\,\dim\ker S^*\,=\,\infty
     \end{equation}
     that replaces \eqref{eq:Aunital} above. As is evident by the arguments that follow, the case of non-zero finite deficiency index is trivially covered as well.

   As emphasised in Section \ref{sec:intro}, the existence of self-adjoint extensions of $S$ with arbitrary point spectrum in the gap is an already established result: the virtue of the new discussion presented now is to re-obtain and re-understand it by \emph{identifying} the actual operator-theoretic \emph{mechanism} of emergence of new essential spectrum \emph{within the general self-adjoint extension scheme of Theorem \ref{thm:VB-representaton-theorem_Tversion2}}.

   It is instructive for comparison purposes to briefly revisit the classical line of reasoning. It originates from ideas and stimuli by Albeverio, leading to the first complete result \cite[Theorem 2.2]{Brasche-Neidhardt-Weidmann-1992-1993} in the infinite deficiency index case by Brasche, Neidhardt, and Weidmann, and later refinements and reformulations by Albeverio, Brasche, and Neidhardt \cite[Proposition 3.1]{Albeverio-Brasche-Neidhardt-JFA1998}, by Brasche \cite[Corollary 30]{Brasche2001_ContMath2004}, and by Albeverio, Konstantinov, and Koshmanenko \cite[Theorem 3.5]{Albe-Dudkin-Konst-Koshm-2005}.

   The idea is: 
   \begin{enumerate}
    \item[(i)] thanks to the infinity of the deficiency index of $S$, one shows that it is possible to iteratively pick countably many orthonormal eigenvectors $v_1,v_2,v_3,\dots$ of $S^*$ with eigenvalues, respectively, $\lambda_1,\lambda_2,\lambda_3,\dots$ arbitrarily chosen within the gap of $S$, and correspondingly to define
   \begin{equation}\label{eq:defSprime}
     \begin{split}
      \mathcal{D}'\,&:=\,\mathcal{D}(\overline{S})\dotplus\bigoplus_n\mathrm{span}\{v_n\}\,, \\
      S'\,&:=\,S^*\upharpoonright\mathcal{D}'\,;
     \end{split}
   \end{equation}
   \item[(ii)]    by construction, 
   $S'$ is a closed symmetric extension of $S$; furthermore, $\bigoplus_n\mathrm{span}\{v_n\}$ is a reducing subspace for $S'$ and therefore, with respect to the Hilbert space orthogonal direct sum
   \begin{equation}
    \cH\,=\,\mathrm{span}\{v_1,v_2,v_3,\dots\}^\perp\:\oplus\:\Big(\bigoplus_n\mathrm{span}\{v_n\}\Big)\,,
   \end{equation}
   one may re-write
   \begin{equation}\label{eq:Sprimedirectsum}
    S'\,=\,\widehat{S} \oplus \Big(\bigoplus_n S_n \Big)\,,\quad S_n\,:=\,S^*\upharpoonright\mathrm{span}\{v_n\}\,,
   \end{equation}
   where $\widehat{S}$ is the densely defined and closed symmetric operator in the Hilbert subspace $\mathrm{span}\{v_1,v_2,v_3,\dots\}^\perp$ defined by 
   \begin{equation}\label{eq:domainShat}
    \begin{split}
     \mathcal{D}(\widehat{S})\,&:=\,\mathcal{D}(S')\cap \mathrm{span}\{v_1,v_2,v_3,\dots\}^\perp\,=\,\mathcal{D}(\overline{S})\cap \mathrm{span}\{v_1,v_2,v_3,\dots\}^\perp\,, \\
     \widehat{S} &:=\,S^*\upharpoonright \mathcal{D}(\widehat{S})\,;
    \end{split}
   \end{equation}
   \item[(iii)] last, one can prove that $\widehat{S}$ is gapped, with at least the same gap as $S$, hence it admits a distinguished self-adjoint extension $\widehat{S}_\mathrm{D}$ in $\mathrm{span}\{v_1,v_2,v_3,\dots\}^\perp$, with everywhere defined and bounded inverse, and with spectral gap given by the gap of $\widehat{S}$, thus a spectral gap containing all the $\lambda_n$'s; since each $S_n$ is obviously self-adjoint in $\mathrm{span}\{v_n\}$, the operator
   \begin{equation}\label{eq:ShatD}
    \widehat{S}_\mathrm{D}\oplus \Big(\bigoplus_n S_n \Big)
   \end{equation}
   is self-adjoint in $\cH$, it extends $S$, and it admits all the $\lambda_n$'s as eigenvalues.
   \end{enumerate}
  (We refer to the above-mentioned literature for the somewhat laborious, and occasionally also subtle, justifications of the above scheme.)

   In the classical argument outlined above one exploits the possibility of expressing the closed symmetric extension $S'$ of $S$ in the reduced form \eqref{eq:Sprimedirectsum}: the self-adjoint extension then only takes place in the Hilbert subspace $\mathrm{span}\{v_1,v_2,v_3,\dots\}^\perp$, and there the distinguished extension is taken. This, however, obfuscates how the resulting self-adjoint extension \eqref{eq:ShatD} of $S$ is classified as an extension $S_T$ within the standard extension scheme of Theorem \ref{thm:VB-representaton-theorem_Tversion2}. The reason for that is further elucidated in the following Remark.

   \begin{remark}
    Observe the difference, for concreteness in the prototypical case of unit deficiency index -- so, for the present Remark, let us return to the assumptions \eqref{eq:A1}-\eqref{eq:Aunital}. Following the scheme above, for arbitrary $\lambda$ in the gap of $S$ and for arbitrary eigenvector $v$ of $S^*$ with eigenvalue $\lambda$ one considers the extension $S'$ of $S$ defined by \eqref{eq:defSprime} as the restriction of $S^*$ to
    \begin{equation}\label{eq:oplussum}
     \mathcal{D}(\overline{S})\dotplus\ker(S^*-\lambda\mathbbm{1})\,,
    \end{equation}
    where $\ker(S^*-\lambda\mathbbm{1})=\mathrm{span}\{v\}$.
    This is precisely the operator $S_\lambda$ of Section \ref{sec:preparation-unital}. There, the way to recognise $S_\lambda$ as a member $S_T$ (in fact, $S^{[\beta]}$, in the notation therein) of the family of self-adjoint extensions of $S$ was to operate on the \emph{second} summand of the direct, non-orthogonal sum \eqref{eq:oplussum}, decomposing $v=\overline{S}x+z$ along $\mathrm{ran}\overline{S}$ and $\ker S^*$ respectively: this eventually led to the explicit identification of the Birman operator $T$ acting on $\ker S^*$. In the scheme \eqref{eq:defSprime}-\eqref{eq:ShatD}, instead, one operates on the \emph{first} summand of \eqref{eq:oplussum}, projecting $\mathcal{D}(\overline{S})$ onto $\ker(S^*-\lambda\mathbbm{1})^\perp$, as \eqref{eq:domainShat} clearly shows. This is not natural, though, if one aims at identifying the Birman parameter $T$ that labels the considered self-adjoint extension.    
   \end{remark}

   The discussion laid down so far finally leads us to the announced identification of the self-adjoint extension mechanism, within the extension scheme of Theorem \ref{thm:VB-representaton-theorem_Tversion2}, for the construction of self-adjoint extensions of gapped operators with arbitrary point spectrum in the gap (Theorem \ref{thm:essspecinfdefind} below).

   Observe that no separability restriction will be needed here on the Hilbert space $\cH$, unlike the above mentioned analyses \cite{Brasche-Neidhardt-Weidmann-1992-1993,Albeverio-Brasche-Neidhardt-JFA1998,Brasche-JOT2000,Brasche2001_ContMath2004,Kondej-2001,Albeverio-Brasche-Malamud-Neidhardt-JFA2005,Albe-Dudkin-Konst-Koshm-2005,Konstantinov-2005,Albe-Dudkin-Konst-Koshm-2007,Behrndt-Khrabustovskyi-2021}. 
   It will be not required either that the distinguished self-adjoint extension $S_{\mathrm{D}}$ of $S$ referred to in assumption \eqref{eq:A1} has compact resolvent, as instead crucially needed in \cite{Behrndt-Khrabustovskyi-2021}.


   The preparation is simple: under assumptions \eqref{eq:A1} and \eqref{eq:Ass-inf-def-ind}, let $(\lambda_n)_{n\in\mathcal{N}}$ be an arbitrary finite or countably infinite collection of (possibly repeated) points inside the gap of $S$, and let $(v_n)_{n\in\mathcal{N}}$ be an orthonormal system in $\mathcal{D}(S^*)$ with
   \begin{equation}\label{eq:Sstarvnlambdavn}
    S^*v_n\,=\,\lambda_n v_n\,, \qquad n\in\mathcal{N}\,.
   \end{equation}
   As mentioned already (see also \cite[Theorem 2.2]{Brasche-Neidhardt-Weidmann-1992-1993} or \cite[Proposition 3.1]{Albeverio-Brasche-Neidhardt-JFA1998}), this choice is always possible by induction, owing to the fact that $\dim\ker(S^*-\lambda_n\mathbbm{1})=\infty$. Then, in analogy to the unital deficiency index case, set 
   \begin{equation}\label{eq:def_un}
    u_n\,:=\,(S_\mathrm{D}-\lambda_n\mathbbm{1})S_\mathrm{D}^{-1}v_n\,, \qquad n\in\mathcal{N}\,,
   \end{equation}
   where $S_\mathrm{D}$ is the distinguished self-adjoint extension of $S$ with everywhere defined and bounded inverse in $\cH$. Clearly, the $u_n$'s are all non-zero, as
   \begin{equation}\label{eq:vnfromun}
    v_n\,=\,S_\mathrm{D}(S_\mathrm{D}-\lambda_n\mathbbm{1})^{-1}u_n\,;
   \end{equation}
   indeed, since $\lambda_n$ belongs to the gap of $S$, $S_\mathrm{D}-\lambda_n\mathbbm{1}$ has everywhere defined and bounded inverse in $\cH$.
   
   \begin{lemma}\label{lem:unvn}
    $(u_n)_{n\in\mathcal{N}}$ is a collection of linearly independent non-zero vectors in $\ker S^*$. The linear map induced by the correspondence $v_n\mapsto u_n$ is a bijection between the Hilbert subspaces
    \begin{equation}\label{eq:defVU}
    V\,:=\,\bigoplus_{n\in\mathcal{N}}\mathrm{span}\{v_n\}\qquad\textrm{and}\qquad U\,:=\,\overline{\,\mathrm{span}\{u_n\,|\,n\in\mathcal{N}\}}\,.
   \end{equation}
   \end{lemma}

   \begin{proof} From \eqref{eq:Sstarvnlambdavn} and \eqref{eq:def_un} one has
   \[
    \begin{split}
      S^* u_n\,=\,S^*(S_\mathrm{D}-\lambda_n\mathbbm{1})S_\mathrm{D}^{-1}v_n\,&=\,S^*(v_n-\lambda_n S_\mathrm{D}^{-1}v_n ) \\
      &=\,S^*v_n-\lambda_n v_n \,=\,0\,,
    \end{split}
   \]
   that is, $u_n\in\ker S^*$. As for the linear independence of the $u_n$'s, observe first that for any selection of complex coefficients $(c_n)_{n\in\mathcal{N}}$ for which the following series converge as elements of $\mathcal{D}(S^*)$, one has 
   \[
    \begin{split}
     \sum_{n\in\mathcal{N}}c_n u_n\,&=\,\sum_{n\in\mathcal{N}}c_n (S_\mathrm{D}-\lambda_n\mathbbm{1})S_\mathrm{D}^{-1}v_n \\
     &=\,\sum_{n\in\mathcal{N}}c_n v_n - \sum_{n\in\mathcal{N}}c_n \lambda_n S_\mathrm{D}^{-1} v_n\,.
    \end{split}
   \]
   If $\sum_{n\in\mathcal{N}}c_n u_n=0$, then
   \[\tag{*}\label{eq:series}
    \sum_{n\in\mathcal{N}}c_n v_n\,=\,S_\mathrm{D}^{-1}\Big(\sum_{n\in\mathcal{N}}c_n \lambda_n v_n\Big)\,.
   \]
   The latter is an identity in $\mathcal{D}(S_\mathrm{D})\cap\big(\bigoplus_{n\in\mathcal{N}}\mathrm{span}\{v_n\}\big)$. In view of the orthonormality of the $v_n$'s one can re-express such subspace as 
   \[
    \begin{split}
     \mathcal{D}(S_\mathrm{D})\cap\Big(\bigoplus_{n\in\mathcal{N}}\mathrm{span}\{v_n\}\Big)\,&=\,\bigoplus_{n\in\mathcal{N}}\big( \mathcal{D}(S_\mathrm{D})\cap\mathrm{span}\{v_n\}\big) \\
     &\subset\,\bigoplus_{n\in\mathcal{N}}\big( \mathcal{D}(S_\mathrm{D}-\lambda_n\mathbbm{1})\cap\ker(S^*-\lambda_n\mathbbm{1})\big)\,.
    \end{split}
   \]
   However, $\mathcal{D}(S_\mathrm{D}-\lambda_n\mathbbm{1})\cap\ker(S^*-\lambda_n\mathbbm{1})=\{0\}$ for each $n$, owing to the direct decomposition $\mathcal{D}(S^*)=\mathcal{D}(S_\mathrm{D}-\lambda_n\mathbbm{1})\dotplus\ker(S^*-\lambda_n\mathbbm{1})$ (see \eqref{eq:domainadjoint} above). Therefore, each side in \eqref{eq:series} must vanish. In particular, $\sum_{n\in\mathcal{N}}c_n v_n=0$. As the $v_n$'s constitute an orthonormal system, necessarily $c_n=0$ $\forall n\in\mathcal{N}$. This establishes the linear independence of the $u_n$'s. The last part of the thesis is now obvious, as \eqref{eq:def_un} and \eqref{eq:vnfromun} invert each other.    
   \end{proof}

   Next, let us complete the general preparation by defining $\widetilde{T}$ as the \emph{self-adjoint} operator in the Hilbert subspace $V$ obtained by linear extension and operator closure from the matrix elements
   \begin{equation}
    \langle v_m,\widetilde{T} v_n\rangle\,:=\,\lambda_n\langle v_m,v_n\rangle-\lambda_m\lambda_n\langle v_m, S_\mathrm{D}^{-1} v_n\rangle
   \end{equation}
   with respect to the orthonormal basis $(v_n)_{n\in\mathcal{N}}$ of $V$, and by defining $T$ as the corresponding self-adjoint operator in the Hilbert subspace $U\subset\ker S^*$ via the bijection $V\cong U$ considered in Lemma \ref{lem:unvn}. In practice,
   \[
    \widetilde{T}\,=\,\bigoplus_{n\in\mathcal{N}}\lambda_n P_n + {\sum_{n,m\in\mathcal{N}}}^{\!\!\!\!\oplus}\:\lambda_m\lambda_n\, P_m S_\mathrm{D}^{-1} P_n
   \]
   in the sense of operator direct sums, where the $P_n$'s are the orthogonal projections, respectively, onto the $\mathrm{span}\{v_n\}$'s, and correspondingly
   \[
    \langle u_m, T u_n\rangle\,=\,\langle v_m,\widetilde{T} v_m\rangle\,,
   \]
   or, with any of the following equivalent expressions,
   \begin{equation}\label{eq:Tmatrixelements}
   \begin{split}
    \langle u_m, T u_n\rangle\,&=\,\lambda_n\langle v_m,v_n\rangle-\lambda_m\lambda_n\langle v_m, S_\mathrm{D}^{-1} v_n\rangle \\
    &=\,\lambda_n \langle (S_\mathrm{D}-\lambda_m\mathbbm{1})S_\mathrm{D}^{-1} v_m,v_n\rangle \\
    &=\,\lambda_n \langle u_m, S_\mathrm{D}(S_\mathrm{D}-\lambda_n\mathbbm{1})^{-1}u_n \rangle \\
    &=\,\lambda_n \langle u_m,v_n\rangle\,.
   \end{split}   
   \end{equation}

   \begin{remark}
    The present definition of $T$ reproduces formula \eqref{eq:beta-lambda-1} of the special case of unital deficiency index (Proposition \ref{lem:SlambdaSbeta}). Indeed, for that case \eqref{eq:Tmatrixelements} reads
    \[
     \langle u, T u \rangle \,=\,\lambda \langle u, S_\mathrm{D}(S_\mathrm{D}-\lambda\mathbbm{1})^{-1}u \rangle 
    \]
    for $\lambda$ in the gap of $S$ and $u=(S_\mathrm{D}-\lambda\mathbbm{1})S_\mathrm{D}^{-1}v\in\ker S^*$ for given \emph{normalised} $v\in\ker(S^*-\lambda\mathbbm{1})$; since $T$ is here the multiplication by $\beta$, one then has 
    \[
     \beta \,=\,\frac{\langle u, T u\rangle}{\:\|u\|^2}\,=\,\lambda\,\frac{\,\langle u, S_\mathrm{D}(S_\mathrm{D}-\lambda\mathbbm{1})^{-1}u \rangle\,}{\:\|u\|^2}\,,
    \]
    consistently with \eqref{eq:beta-lambda-1}.    
   \end{remark}

   \begin{theorem}\label{thm:essspecinfdefind}
    Under the assumptions \eqref{eq:A1} and \eqref{eq:Ass-inf-def-ind}, and for an arbitrary finite or countably infinite collection $(\lambda_n)_{n\in\mathcal{N}}$ of points in the spectral gap of $S$, the self-adjoint extension $S_T$ of $S$ labelled, according to the general extension scheme \eqref{eq:ST-2}, by the Birman parameter $T$ defined above has all the $\lambda_n$'s as eigenvalues. If $\lambda$ is any of such values, then its multiplicity as an eigenvalue is no less than the multiplicity of the $\lambda_n$'s being equal to $\lambda$.    
   \end{theorem}

   \begin{proof}
    Let $(v_n)_{n\in\mathcal{N}}$ be the orthonormal system selected in $\mathcal{D}(S^*)$ for the construction of $(u_n)_{n\in\mathcal{N}}$, $U$, $V$, and $T$ from \eqref{eq:def_un}, \eqref{eq:defVU}, and \eqref{eq:Tmatrixelements}. Decompose
    \[
     v_n\,=\,\overline{S}x_n+z_n
    \]
    for $x_n\in\mathcal{D}(\overline{S})$ and $z_n\in\ker S^*$ uniquely determined via the orthogonal direct sum $\cH=\,\ran\,\overline{S}\oplus\ker S^*$. Next, define
    \[
     \begin{split}
      f_n\,&:=\,\lambda_n x_n\,\in\,\mathcal{D}(\overline{S})\,, \\
      W\,&:=\,\ker S^*\ominus U\,, \quad\mathrm{i.e.,}\quad \ker S^*=U\oplus W\,, \\
      w_n\,&:=\,\lambda_n P_W z_n \,\in\,W\,,
     \end{split}
    \]
   where $P_W$ denotes the orthogonal projection onto $W$. With these definitions, consider the element $f_n+S_{\mathrm{D}}^{-1}(Tu_n+w_n)+u_n\in\mathcal{D}(S_T)$, where $S_T$ is the self-adjoint extension of $S$ parametrised by the present Birman operator $T$, according to the general classification \eqref{eq:ST-2} of Theorem \ref{thm:VB-representaton-theorem_Tversion2}. One has 
   \[
    \begin{split}
     f_n+& S_{\mathrm{D}}^{-1}(Tu_n+w_n)+u_n \\
     &=\,\lambda_n x_n + S_{\mathrm{D}}^{-1}(Tu_n+w_n)+ (S_\mathrm{D}-\lambda_n\mathbbm{1})S_\mathrm{D}^{-1}v_n \\
     &=\,\lambda_n S_{\mathrm{D}}^{-1}(\overline{S}x_n)+ S_{\mathrm{D}}^{-1}(Tu_n+w_n)+v_n-\lambda_nS_{\mathrm{D}}^{-1} v_n \\
     &=\,-\lambda_n S_{\mathrm{D}}^{-1}z_n+ S_{\mathrm{D}}^{-1}(Tu_n+w_n)+v_n  \\
     &=\,S_{\mathrm{D}}^{-1}y_n+v_n\,,
    \end{split}
   \]
   where 
   \[
    y_n\,:=\,Tu_n+w_n-\lambda_n z_n\,\in\,\ker S^*\,.
   \]
   By inspection, $y_n=0$. Indeed, along $U$ one has, for any $m\in\mathcal{N}$, 
   \[
    \begin{split}
     \langle u_m,x_n\rangle\,&=\,\langle u_m, Tu_n\rangle-\lambda_n\langle u_m,z_n\rangle \\
     &=\,\langle u_m, Tu_n\rangle-\lambda_n\langle u_m,v_n\rangle\,=\,0
    \end{split}
   \]
   (having used that $u_m\perp w_n$ in the first identity, $z_n=v_n-\overline{S}x_n$ and $u_n\perp \overline{S}x_n$ in the second identity, and the last of \eqref{eq:Tmatrixelements} in the final step), whereas along $W$ one has
   \[
    P_W y_n\,=\,w_n-\lambda_n P_W z_n\,=\,0\,.
   \]
   Thus,
   \[
    v_n\,=\,f_n+ S_{\mathrm{D}}^{-1}(Tu_n+w_n)+u_n\,\in\,\mathcal{D}(S_T)\,,
   \]
   and $S_Tv_n=S^*v_n=\lambda_n v_n$, which proves that indeed the self-adjoint extension $S_T$ admits all the $\lambda_n$'s as eigenvalues. The statement on the multiplicity is then obvious, since the eigenvectors $v_n$'s constitute an orthonormal system.  
   \end{proof}

   \begin{remark}
    In the unital deficiency index case of Sect.~\ref{sec:preparation-unital} the present identification $y_n=Tu_n+w_n-\lambda_n z_n$ reads $y=Tu-\lambda z=\beta u-\lambda z$ in the notation therein, since $\ker S^*=\mathrm{span}\{u\}=U$ and hence $W=\ker S^*\ominus U=\{0\}$. As observed, $\beta u=\lambda z$, whence $y=0$.   
   \end{remark}

   The above construction clearly covers also the case when the gapped operator $S$ has \emph{finite} deficiency index and one produces a self-adjoint extension with finitely many given eigenvalues within the gap. As mentioned, this special case was the original result of Kre{\u\i}n \cite[Theorem 23]{Krein-1947}, but it is re-obtained here in a considerably more economic way then Kre{\u\i}n's ingenious but involved proof. This cheapness of course relies on a highly sophisticated toolbox, namely Theorem \ref{thm:VB-representaton-theorem_Tversion2} providing the general classification of all self-adjoint extensions of $S$, which was unavailable at the time of \cite{Krein-1947} and was only completed by Vi\v{s}ik \cite{Vishik-1952} and Birman \cite{Birman-1956} in the following years.

   For the same reason, Theorem \ref{thm:essspecinfdefind} above provides a much more direct construction of the desired self-adjoint extension also in the general case of \emph{infinite} deficiency index, as compared to the classical reasoning outline at the beginning of this Section.

   \begin{corollary}
    Let  $S$ be a densely defined, symmetric, gapped operator on an infinite-dimensional Hilbert space $\cH$, and assume that $S$ has infinite deficiency index. For any arbitrary closed subset $K$ of the closure of the gap of $S$ there is a self-adjoint extension of $S$ whose essential spectrum contains $K$.
   \end{corollary}

   \begin{proof}
    It is clearly non-restrictive to assume the gap $(a,b)$ of $S$ to include $0$. Conditions \eqref{eq:A1} and \eqref{eq:Ass-inf-def-ind} are thus satisfied and it is possible to construct the above self-adjoint extension $S_T$ of $S$ relative to a collection $(\lambda_n)_n$ of points that form a dense of $K\cap(a,b)$. Then $S_T$ admits all such $\lambda_n$'s as eigenvalues, implying that $K\subset\sigma_{\mathrm{ess}}(S_T)$.    
   \end{proof}

   It is worth remarking, in view of the construction above, that the arbitrary portion $K$ of new essential spectrum for $S_T$ inside the closure of the gap $(a,b)$ can be of either type: points of $\sigma(S_T)$ that are accumulation points of actual eigenvalues, as well as isolated eigenvalues of $S_T$ with infinite multiplicity.

   As mentioned in the background discussion (Section \ref{sec:intro}), a very recent instance of emergence, by self-adjoint extension, of one single isolated point of essential spectrum within the gap is the three-dimensional Dirac operator with critical shell interaction supported on the union of finitely many disjoint spheres \cite{Benhellal-Pankrashkin-2022}. That is precisely the value $-m\mu/\varepsilon$ for given mass, electrostatic shell, and Lorenz shell parameters, respectively, $m\geqslant 0$ and $\varepsilon,\mu\in\mathbb{R}$. In \cite{Benhellal-Pankrashkin-2022} such a value of essential spectrum is identified by indirect means that do not elucidate its emergence within the general self-adjoint extension scheme. It would be of valuable interest to investigate whether the self-adjoint Dirac operator at criticality, with singular perturbation on spherical shells, is indeed of the form $S_T$ of Theorem \ref{thm:essspecinfdefind}.

%

 \def\cprime{$'$}

\end{document}